\documentclass[letterpaper, 10 pt, conference]{ieeeconf}  % Comment this line out if you need a4paper

\IEEEoverridecommandlockouts            
\usepackage{amssymb}
\usepackage{amsmath}
\usepackage{graphicx}
\usepackage[colorlinks]{hyperref}
\usepackage[colorinlistoftodos]{todonotes}
\usepackage{verbatim}
\usepackage{algorithm,algcompatible,amsmath}
\usepackage{mathtools}

\usepackage{hyperref}
\hypersetup{
    colorlinks=true,
    linkcolor=red,
    filecolor=magenta,      
    urlcolor=cyan,
}
\usepackage{graphicx}

\newtheorem{theorem}{\bf Theorem}

\newtheorem{definition}{\bf Definition}

\newtheorem{lemma}{\bf Lemma}

\def\QED{~\rule[-1pt]{5pt}{5pt}\par\medskip}

\DeclareMathOperator*{\argminB}{argmin} 
\title{\LARGE \bf
Linearly-Solvable Mean-Field Approximation for Multi-Team Road Traffic Games 
}

\author{Ali Reza Pedram \and  Takashi Tanaka% <-this % stops a space
\thanks{AR. Pedram is with Mechanical Engineering,
        University of Texas at Austin, TX, USA
        {\tt\small apedram@utexas.edu}}%
\thanks{T. Tanaka is with Department of Aerospace Engineering and Engineering
Mechanics, University of Texas at Austin, TX, USA
        {\tt\small ttanaka@utexas.edu}}%
}
\begin{document}

\maketitle
\thispagestyle{empty}
\pagestyle{empty}

%%%%%%%%%%%%%%%%%%%%%%%%%%%%%%%%%%%%%%%%%%%%%%%%%%%%%%%%%%%%%%%%%%%%%%%%%%%%%%%%
\begin{abstract}
We study the traffic routing game among a large number of selfish drivers over a traffic network. We consider a specific scenario where the strategic drivers can be classified into teams, where drivers in the same team have identical payoff functions. An incentive mechanism is considered to mitigate congestion, where each driver is subject to dynamic tax penalties. We explore a special case in which the tax is affine in the logarithm of the number of drivers selecting the same route from each team.  It is shown via a mean-field approximation that a Nash equilibrium in the limit of a large population can be found by linearly solvable algorithms. 
\end{abstract}

%%%%%%%%%%%%%%%%%%%%%%%%%%%%%%%%%%%%%%%%%%%%%%%%%%%%%%%%%%%%%%%%%%%%%%%%%%%%%%%%
\section{Introduction}

Transportation is a major energy consuming sector in the United States, accounting for 28\% of total use and 26\% of green-house emissions \cite{jang2018simulation}. The economic loss due to traffic congestion is also significant; in 2014, it was estimated to account for $\$$160 billion \cite{US}. Recently, new stakeholders such as ride-hailing companies and mobile routing apps are reported to have substantial impacts on urban traffic networks \cite{herrera2010evaluation}. Therefore, novel theoretical frameworks for traffic analysis and control in which ride-hailing companies act as decision-makers are urgently needed.

%Transportation is a major energy consuming sector and greenhouse gas emitter in the USA,  accounting for 28\% of total use and 26\%  of emissions, respectively  . To be more specific, the economic loss due to traffic congestion is significant, estimated to account for the loss of 160 billion dollars in 2014 . Equally important, the emergence of mobile routing apps (e.g. Google, Waze, Apple, etc.) and the ride-hailing servicesare reported to have a key role in traffic management . However, to address the challenges of today's transportation system, modifications and advancement are required to fully exploit the potentials of these apps. Thus, novel technologies and protocols for traffic analysis and control are needed.  

A traffic system can be analyzed by the theory of dynamic games, where the non-cooperative drivers compete over the shared network \cite{krichene2018learning, fisac2018hierarchical}. Strategic behaviour of individuals often results in a game theoretic equilibrium that is not necessarily socially optimal. In the literature, the loss of optimality (efficiency loss) is commonly measured in terms of \textit{the price of anarchy (PoA)} \cite{youn2008price, mehr2018can}. POA can be improved by an appropriate incentive design; \cite{brown2018optimal,cole2003pricing} considered dynamic toll mechanisms, which are assumed to be operated by a Traffic System Operator (TSO).

%To quantify such an efficiency loss, \textit{the price of anarchy (PoA)} have been widely used by many researchers \cite{youn2008price,mehr2018can}. To mitigate the efficiency loss due to strategic behaviors (as assessed by PoA), \cite{brown2018optimal,cole2003pricing} considered the situations in which Traffic System Operator (TSO) utilizes toll mechanisms.

It should be noted that incentive mechanisms operated by TSOs are restricted by the information available to them. 
 % designing mechanisms for traffic flow applications is highly dependent on the information available to the TSO.
For example, \cite{karakostas2004edge,fleischer2004tolls} consider fixed taxation for each route based on the full characterization of the network topology and user demands. In contrast, \cite{sandholm2002evolutionary} proposed a tax structure which is independent of this information. A complete review of the design and evaluation of road network pricing schemes is provided in \cite{tsekeris2009design}. 

Traffic systems in reality typically consist of a large number of  strategic drivers and
their game-theoretic analysis is often computationally challenging. Therefore, an incentive mechanism should ideally be designed in such a way that (i) its implementation is computationally simple so that it is scalable to large networks, and (ii) there exists a simple computational procedure to find an equilibrium in the resulting dynamic game. If the later condition is not satisfied, it is questionable that drivers in reality can ever take the equilibrium strategy. 

%Therefore, an effective price mechanism should have the following characteristics. First, the mechanism should be  computationally simple and scalable so that it will be applicable to complex and large-scale networks. Furthermore, the mechanism should yield a computable equilibrium among strategic drivers to ensure that the selfish drivers can efficiently compute their optimal actions.

To simplify the analysis of routing games with a large number of drivers, a common approach is to adopt a macroscopic perspective, where the density of the vehicles is modeled rather than the dynamics of individual vehicles \cite{bellomo2011modeling}.  In a similar spirit, recently the framework of Mean Field Game (MFG)  \cite{chevalier2015micro} was applied to compute the traffic density propagation induced from the interactions of many independent strategic drivers. Despite the substantial progress in MFGs, the number of papers utilize the MFG theory for macroscopic traffic modelling is still limited. Some recent studies have explored MFGs with multiple classes of drivers \cite{tembine2011mean}, which is appropriate to model traffic games. Additionally, some papers have considered scenarios with a major player and a large number of minor players \cite{huang2012mean}.

In MFG setting, the traffic flow is modeled as a fluid whose behavior can be obtained by solving  Hamilton-Jacobi-Bellman (HJB) equation coupled with the standard conservation law applied to the vehicles, refered to as Fokker-Planck-Kolmogrov (FPK) equation. In \cite{festa2018mean}, an MFG approach is taken to address multi-lane traffic
management, using a semi-Lagrangian scheme to solve an approximation of a coupled HJB-FPK.

In \cite{tanaka2018linearly,tanaka2019linearly}, a discrete-time dynamic
stochastic game over an urban network is studied, wherein at each intersection, each strategic driver selects one of the outgoing links randomly with respect to his/her mixed policy. A tax mechanism is offered which yields a linearly solvable game under the assumptions that all drivers have common origin and destination (O/D) and travel costs. Then, the backward HJB and forward FPK equations can be solved independently and thus a MFE is obtained through a linearly-solvable set of equations. This result \cite{tanaka2018linearly} has an important computational advantage with respect to previous results \cite{chevalier2015micro}  which dealt with a coupled set of HBJ-FPK equations. However, the simplified problem setting in \cite{tanaka2018linearly} with shared pair of (O/D) and travel cost has a limited applicability to practical traffic routing problems. To address this gap, in this paper, we consider more realistic settings, in which there exist multiple teams with different (O/D) pairs and travel costs. As the main technical contribution of this paper, a Nash equilibrium (NE) in the limit of a large number of drivers (MFE) for the multi-team setting is given as the solution to a set of linearly solvable optimal control problems.

This paper is organized as follows: The multi-team road traffic game is set up in Section \ref{probdef}, the behavior of the game in large population limit is studied in Section \ref{ll}. An auxiliary set of control problems is introduced in Section \ref{threeb} which is used to solve the game in the large population limit in Section \ref{threec}. The game with finite number of drivers is considered in Section \ref{gamedef}. 
Numerical studies are summarized in Section \ref{NI} before we conclude in Section \ref{conc}.
%%%%%%%%%%%%%%%%%%%%%%%%%%%%%%%%%%%%%%%%%%%%%%%%%%%%%%%%%%%%%%%%%%%%%%%%%%%%%%%%
\section{Problem Formulation}
\label{probdef}
We assume there is a large number of drivers traveling over a shared network $\mathcal{G} = (\mathcal{V}, \mathcal{E})$, referred to as the traffic graph, where $\mathcal{V} = \{1, 2, \dots, V \}$ is the set of nodes (intersections) and $\mathcal{E} = \{1, 2, \dots, E\}$ is the set of directed edges (links). The set of nodes to which there exist a direct link from $i$ is labeled by $\mathcal{V}(i)$. We  assume that drivers are categorized into $L$ teams, and that drivers from the same team have the same (O/D) pair. The number of drivers in each team is denoted by $N_l \quad \forall l\in\{1, 2, \dots, L\}$. At each time step $t \in \mathcal{T} = \{0, 1, \dots, T-1\}$, the node at which the $n$-th driver from the $l$-th team ($n\leq N_l$) is located is represented by $i_{n,l,t} \in \mathcal{V}$.

\subsection{Routing policy}
For each $n$ and $l$, $i_{n,l,t}$ is a random  process. Let $P_{n,l,t} = \{P^i_{n,l,t}\}_{i\in \mathcal{V}}$ be the probability distribution of the individual driver $n$ in the team $l$ and  over the nodes of the traffic graph at time $t$. At $t = 0$, we assume all drivers in the same team will start from a common probability $P_{l,0}$; however, $i_{n,l,0}$ for each $n$ and $l$ is realized independently with respect to $P_{l,0}$.

For $(n_1,l_1)\neq (n_2,l_2)$, we assume $i_{n_1,l_1,t}$ and $i_{n_2,l_2,t}$ are independent random variables. At every time step, driver $(n,l)$ selects an  action $j_{n,l,t} \in \mathcal{V}(i_{n,l,t})$, and moves to the node $j_{n,l,t}$ at time $t+1$ (i.e. $i_{n,l,t+1} = j_{n,l,t})$.
Each driver would follow a randomized policy (strategy) $Q^i_{n,l,t} = \{Q
^{ij}_{n,l,t}\}_{j\in\mathcal{V}(i)}$, which represents the  probability distribution according to which she chooses her next destination. Let $\Delta^J$ be the $J$-dimensional probability simplex. Then for all drivers $(n,l)$ at time $t$,  $Q_{n,l,t}=\{Q^{i}_{n,l,t}\}_{i\in\mathcal{V}} $  belongs to the space of possible mixed strategies $\mathcal{Q}$, where
\begin{equation*}
\begin{split}
    \mathcal{Q}=&\Big\{ 
\{Q^{i}\}_{i\in \mathcal{V}}: Q^{i} \in \Delta^{|\mathcal{V}(i)|-1}\ \  \forall i \in \mathcal{V}
\Big\}.
\end{split}
\end{equation*}
We assume that each driver  fixes  her strategy $\{Q_{n,l,t}\}_{t\in \mathcal{T}}$, at $t=0$ based on the global knowledge the game's parameters and she will not be able to update it during the game. In this setting, the location probability distribution for driver ($n,l$) is computed recursively by
\begin{equation}
\label{dyn}
P^{j}_{n,l,t+1}=\sum_{i} P_{n,l,t}^i \  Q_{n,l,t}^{ij}, 
\end{equation}
with initial ${P}_{n,l,t}=P_{l,0}$. Further details of the game setup are given below. Note that they are natural generalizations of the setup studied in \cite{tanaka2018linearly} to multi-team scenarios.  
\subsection{Cost Functions and Congestion-reducing Incentives} 
 Each driver is subjected to two different categories of costs-- travel cost and congestion cost.
 \subsubsection{Travel cost}
 Moving from node $i$ to $j$ requires a fixed amount of cost e.g., fuel cost. Let $C^{ij}_{l,t}$ be a given cost for drivers in team $l$ to take this action at time $t$.
 
 \subsubsection{Congestion cost}
 This is the tax cost imposed by the TSO to incentivize drivers to adopt a nominal routing policy, which we assume it is pre-specified. The nominal policy is denoted by ${R}^{ij}_t$, representing the probability of selecting the next destination $j$ at node $i$ at time step $t$. Namely, ${R}^{ij}_t$ satisfies ${R}^{ij}_t \geq 0$ and $\sum_j {R}^{ij}_t=1 $. To penalize the deviation from this reference distribution, TSO introduces the log-population tax mechanism in which $(n,l)$-th driver taking action $j$ at time $t$ from node $i$ will be charged the value of  
\begin{equation}
\label{tax-mech}
\begin{split}
&\pi^{ij}_{n,l,t}=\sum_{m=1}^L a_{lm} \Big(\log( \frac{K^{ij}_{m,t}}{K^i_{m,t}})- \log R^{ij}_t
   \Big),
   \end{split}
\end{equation}
where $K^i_{m,t}$ denotes the population of drivers (including driver $(n,l)$ in the case of $m=l$) from each team in node $i$ at time $t$. Similarily,  $K^{ij}_{m,t}$  represents the number of drivers (from team $m$) who are at node $i$ at time $t$ and willing to take the action $j$. We consider the tax formula (\ref{tax-mech}) mainly because (i) it is a natural extension of the mechanism in \cite{tanaka2018linearly} and (ii) the  resulting game is efficiently solvable, as we will see in Section \ref{three}.

The tax mechanism (\ref{tax-mech}) implies that the driver $(n,l)$ selecting action $j$ will be penalized if the fraction of drivers taking this action is greater than the desired value and it will be rewarded otherwise. The tax penalty (\ref{tax-mech}) is the weighted sum of charges associated with individual teams' deviation from the nominal routing policy. The parameter $ a_{lm} >0$ means that the drivers in team $l$ will be penalized by selecting action $j$ at node $i$ if the road from $i$ to $j$ is going to be overpopulated by the drivers from team $m$. Notice that  $K^i_{m,t}\geq K^{ij}_{m,t} \geq 0$ by construction. We adopt the convention that $\log(\frac{0}{0})=0$. In this setting, each driver is trying to minimize the expected value of the tax cost i.e., $\Pi_{n,l,t}^{ij}\triangleq E [ \pi^{ij}_{n,l,t}|i_{n,l,t}, j_{n,l,t}]$ as well as fixed costs. In summary, each driver is willing to minimize her own charge selfishly by solving:
\begin{equation}
\label{game}
\begin{split}
     \min_{\{Q_{n,l,t}\}_{0}^{T-1}} & \sum_{t=0}^{T-1} \sum_{i, j} P_{n,l,t}^i  Q_{n,l,t}^{ij}(C_{n,l,t}^{ij}+\Pi^{ij}_{n,l,t})\\
\text{s.t } \quad & \sum_{j}  Q_{l,t}^{ij} =1\quad \forall i.
\end{split}
\end{equation}
Notice that (\ref{game}) is a game involving $\sum_{l=1}^L N_l$ players, since the value of  $\Pi_{n,l,t}^{ij}$ depends on the strategies of all other drivers in the system. Analyzing this game is a nontrivial task. However, if the number of drivers is sufficiently large, the game is well-approximated by the one with infinitely many drivers. In the next section, we develope such an approximation via the MFG theory.
\section{ Main Results}
\label{three}
In this section, we study the MFG approximation of the game with a large number of drivers. Then, an efficient algorithm is provided to calculate the MFE of the game.  

\subsection{ Large population limit}
\label{ll}
In \cite{tanaka2018linearly}, for the case of single team,
an explicit expression for $\lim _{N\rightarrow \infty } \Pi_{n,t}$ is obtained. Here, we provide a generalized alternative method to find an analogous limit for the multi-team case.

Due to the indistinguishability of drivers in a team, we can restrict ourselves to the symmetric setups where drivers in a team are sharing a common policy $Q_{l}=\{Q^i\}_{i\in \mathcal{V}, t\in \mathcal{T}, l\in \mathcal{N}_l}$. Denoting the adapted strategy by $Q^{*}_l$, one can calculate the induced probability distribution $P^{*}_l$, based on (\ref{dyn}). In what follows, we study the asymptotic behavior of the congestion charge in the large population limit (i.e., $N = \sum_m N_m \rightarrow \infty$ for the fixed population ratios $\frac{N_l}{N}$).

\begin{lemma}
\label{lemma1}
 If ${P}^{i*}_{l,t}Q^{ij*}_{l,t}>0$ $\forall l\in\{1, 2, \dots, L\}$, then
\begin{equation}
\lim_{ N\rightarrow \infty} \Pi^{ij}_{n,l,t}= \sum_{m=1}^L a_{lm} \log \frac{Q_{m,t}^{ij*}}{R_{t}^{ij}}.    
\end{equation}
\end{lemma}
\begin{proof}
 See appendix \ref{appen1}.
\end{proof}
Therefore, in the large population limit, the optimal response of $(n,l)$-th driver is characterized by:
\begin{align}
 \argminB_{\{Q_{n,l}\}_0^{T-1}}  &\sum_{t=0}^{T-1} \sum_{i,j} P_{n,l,t}^i  Q_{n,l,t}^{ij}(C_{l,t}^{ij}+
 \sum_{m=1}^L a_{lm} \log \frac{Q_{m,t}^{ij*}}{R_{t}^{ij}}) \nonumber \\
 \text{s.t } \quad  &\sum_{j}  Q_{l,t}^{ij} =1\quad \forall i, \label{realgame}
\end{align}
where the probability distribution $P_{n,l,t}$ propagates with respect to (\ref{dyn}). In what follows, we show that $Q^{*}$ is indeed an MFG of multi-team traffic routing game if the minimizer of (\ref{realgame}) is again $Q^{*}_{n,l}$ (That is, $Q^*$ is a fixed point). In the next subsection, we study auxiliary $L$-player game by which such a fixed policy is found.
\subsection{Auxiliary optimization problem }
\label{threeb}
Consider an auxiliary $L$-player game in which each team tries to find her optimum policy $Q^{ij*}_l$, the solution to:
\begin{equation}
\label{aux}
\begin{split}
 \min_{\{Q_l\}_0^{T-1}} & \sum_{t=0}^{T-1} \sum_{i,j} P_{l,t}^i \  Q_{l,t}^{ij}(C_{l,t}^{ij}+\sum_{m=1}^L a_{lm} \log \frac{Q_{m,t}^{ij}}{R_{t}^{ij}}) \\
 \text{s.t } \quad & \sum_{j}  Q_{l,t}^{ij} =1\quad \forall i. 
\end{split}
\end{equation}
Notice that in (\ref{realgame}) the logarithmic term is fixed, in contrast to (\ref{aux}) where it is an optimization variable. For each $t \in \mathcal{T} $, one can introduce value function associated with optimal control problem (\ref{aux}) as: \vspace*{-3mm}
\[V_{l,t} (P_{l,t})\!= \!\!\!\!
\min_{\{Q_l\}_0^{T-1}}\!\! \sum_{\tau=t}^{T-1} \sum_{i, j} P_{l,\tau}^i Q_{l,\tau}^{ij}(C_{l,\tau}^{ij}+\!\sum_{m=1}^L \!a_{lm} \log \frac{Q_{m,t}^{ij}}{R_{t}^{ij}}).\]
The value function satisfies the following coupled Bellman equations:
\begin{align}
\nonumber
 V_{l,t} (P_{l,t})= \min_{{Q}_{l,t}} \bigg\{ &\sum_{i, j}P_{l,t}^i \  Q_{l,t}^{ij}(C_{l,t}^{ij}+\sum_{m=1}^L a_{lm} \log \frac{Q_{m,t}^{ij}}{R_{t}^{ij}}) \\\label{Bellman}
&+ V_{l,t+1}({P}_{l,t+1})\bigg\},
\end{align}
 with $V_{l,T}(P_{l,T})=0 \quad \forall l, P_{l,T}$.
 The first main result of this paper is summarized by the following Theorem \ref{theorem1}. It states that this set of optimal control problems (\ref{aux}) is linearly solvable, which means the solution is obtained by solving a linear system.
\begin{theorem}
\label{theorem1}
Let $ \{\lambda_{l,t}\}_{t\in \mathcal{T}}$  be sequences of V-dimensional vectors, $\{\phi_{l,t}\}_{t\in \mathcal{T}}$ be sequences of V$\times$V-dimensional matrices, and $A$ be an invertible matrix whose $(l,m)$-th element is $a_{lm}$. Then, optimal policies $\{Q^{ij*}_{l,t}\}_{t\in \mathcal{T}}$ can be iteratively calculated by Algorithm \ref{alg1}. Moreover, the value functions for each $t \in \mathcal{T}$ are
  \[ V_{l,t}(P_{l,t})=\sum_{i}P^{i}_{l,t} (-a_{ll}-\lambda^i_{l,t}).\]
\end{theorem}
\begin{proof}
See appendix \ref{appen2}.
\end{proof}
\begin{algorithm}
    \caption{Optimal Strategy Computation for T-stage Road Traffic Game}
   \label{alg1}
  \begin{algorithmic}[1]
   
    \STATE \textbf{Initialize} $\phi^{ij}_{l,T-1}=C^{ij}_{l,T-1}$ for all $i,j \in \mathcal{V}$
    \FOR{$t= T-1, T-2, \dots, 0$}
    
    \STATE for all $i,j \in \mathcal{V}$, compute:
    \[M^{ij}_t=\begin{bmatrix}
    M^{ij}_{1,t}\\
    \vdots\\
    M^{ij}_{L,t}
    \end{bmatrix} := A^{-1} \begin{bmatrix}
    -a_{11}-\phi^{ij}_{1,t}\\
    \vdots\\
    -a_{LL}-\phi^{ij}_{L,t}
    \end{bmatrix} \]
 
     \STATE for every $i \in \mathcal{V}$, compute:  
     \[\Lambda^{i}_{t}=
     \begin{bmatrix}
     \lambda^{i}_{1,t}\\
     \vdots\\
     \lambda^{i}_{L,t}
     \end{bmatrix}:=
     A \begin{bmatrix}
     \log(\sum_j R_{t}^{ij} \exp{M^{ij}_{1,t}})\\
     \vdots\\
     \log(\sum_j R_{t}^{ij} \exp{M^{ij}_{L,t}})
     \end{bmatrix}
      \]
      
     \STATE for every $i,j \in \mathcal{V}$, compute the optimal policies:
      \[ Q^{ij*}_{l,t}=R^{ij}_{l,t} \exp{(M^{ij}_{l,t}-\sum_m A^{-1}_{l,m}\lambda^{i}_{m,t})}\]
      \[ \phi^{ij}_{l,t-1}=C^{ij}_{l,t-1}-a_{ll}-{\lambda}^j_{l,t}\]
      
      \ENDFOR
      \STATE \textbf{return} $\{Q^{ij*}_{l,t}\}_{t\in \mathcal{T}}$, $\{\lambda^{i}_{l,t}\}_{t \in \mathcal{T}}$ and$ \{\phi^{ij}_{l,t}\}_{t \in \mathcal{T}}$.

  \end{algorithmic}
  
\end{algorithm}

We stress that Algorithm~\ref{alg1} is linear in $\phi$ and $\lambda$ and thus the optimal solution can be pre-computed backward in time.

\subsection{Equilibrium in the large population limit}
\label{threec}
Here, we find the equilibrium point of the game with a large number of drivers. It is noteworthy that $\Pi^{ij}_{n,l,t}$, as proved in \cite{tanaka2018linearly}, is independent of the $(n,l)$-th driver's own strategy $Q^{ij}_{n,l,t}$, even for finite $N$. In this game, each strategic driver wants to minimize her cost function, given by 
\begin{equation*}
    J_{(n,l)}(Q_{n,l},Q_{-(n,l)})=\sum_{t=0}^{T-1} \sum_{i, j} P_{n,l,t}^i  Q_{n,l,t}^{ij}(C_{n,l,t}^{ij}+\Pi^{ij}_{n,l,t}),
\end{equation*}
where $Q_{n,l}=\{Q_{n,l,t}\}_{t\in \mathcal{T}}$ denotes her policy and
$\{\Pi^{ij}_{n,l,t}\}_{t\in \mathcal{T}}$ is determined by
other drivers' policy $Q_{-(n,l)}=\{Q_{\tilde{n},\tilde{l},t}\}_{t\in \mathcal{T}, \forall (\tilde{n},\tilde{l})\neq (n,l)}$. 
\vspace*{2mm}
\begin{definition}
\label{def1}
The set of  strategies $\{Q^{NE}_{n,l}\}_{l=1, 2, \dots, L, n=1, 2, \dots, n_l}$ 
is said to be a NE if the following
inequalities hold.
\begin{equation*}
    \begin{split}
        J_{(n,l)}(Q_{n,l},Q^{NE}_{-(n,l)})
        \geq J_{(n,l)}(Q^{NE}_{n,l},Q^{NE}_{-(n,l)}) \quad \forall (n,l)\\
    \end{split}
\end{equation*}
\end{definition}
\vspace{2mm} 
As a candidate of NE in the limit of $N \rightarrow \infty$ , we consider $Q^{ij*}_{l,t}$, the solution to (\ref{aux}). To validate this guess, let's consider all drivers except $(n,l)$ adopt $Q^{ij*}_{l,t}$ and see if it statisfies the condition  in Definition~\ref{def1}. One can again apply dynamic programming and define the value function for (\ref{realgame}) as:
\begin{equation*}
\label{valuedef2}
\begin{split}
&\tilde{V}_{n,l,t} (P_{n,l,t})=\\
&\min_{ \{Q_{n,l}\}_0^{T-1}} \sum_{\tau=t}^{T-1} \sum_{i, j} P_{n,l,\tau}^i Q_{n,l,\tau}^{ij}(C_{l,\tau}^{ij}+\sum_{m=1}^L a_{lm} \log \frac{Q_{m,t}^{ij*}}{R_{t}^{ij}}),
\end{split}
\end{equation*}
and the associated Bellman equation will be
\begin{align}
\nonumber
& \tilde{V}_{n,l,t} (P_{n,l,t})=\min_{{Q}_{n,l,t}} \bigg\{\sum_{i, j}P_{n,l,t}^i  Q_{n,l,t}^{ij}(C_{l,t}^{ij}\\\label{Bellman2}
& +\sum_{m=1}^L a_{lm} \log \frac{Q_{m,t}^{ij*}}{R_{t}^{ij}})+ \tilde{V}_{n,l,t+1}({P}_{n,l,t+1})\bigg\}.
\end{align}
Note that $V$ and $\tilde{V}$ are value functions associated with different optimal control problems.
\begin{lemma}
\label{lemma2}
For all $t \in \mathcal{T}$ and any given $P_{n,l,t}$, $\tilde{V}_{n,l,t}(P_{n,l,t})=V_{l,t}(P_{n,l,t})$. 
Moreover, an arbitrary policy $\{{Q}_{n,l,t}\}_{t \in \mathcal{T}} \in \mathcal{Q}$ is an optimal solution to (\ref{realgame}).
\end{lemma}
\begin{proof} Starting from $t=T-1$, by substituting the explicit expression for $Q_{m,t}^{ij*}$  from (\ref{q-star}), we have:
\begin{subequations}
\label{Bellman_3}
\begin{align}
\nonumber
&\tilde{V}_{n,l,t}({P}_{n,l,t})\\ \nonumber
&=\min_{Q_{n,l,T-1}} \! \sum_{i, j} P_{n,l,T-1}^i  Q_{n,l,T-1}^{ij}(\phi_{l,T-1}^{ij}\!\!+\!\!
\sum_{m=1}^L\! a_{lm}\! \log \frac{Q_{m,t}^{ij*}}{R_{T-1}^{ij}}\!)\\
\nonumber
&=\min_{Q_{n,l,T-1}} \! \sum_{i, j} P_{n,l,T-1}^i  Q_{n,l,T-1}^{ij}(- a_{ll}-\lambda^i_{l,T-1})\\ \nonumber
&=\min_{Q_{n,l,T-1}} \! \sum_{i} P_{n,l,T-1}^i (- a_{ll}-\lambda^i_{l,T-1}) \underbrace{\sum_j Q_{n,l,T-1}^{ij}}_{=1}\\ \nonumber
&=\ V_{l,t} ({P}_{n,l,T-1})
\end{align}
\end{subequations}
 Similarly, the proof can be repeated inductively for all $t$. Since the final expression does not depend on $Q_{n,l,t}$, any allowable policy is a minimizer.
\end{proof}

This result (all feasible solutions are optimal) can be interpreted as the Wardrop's first principle, which states that at equilibrium,\textit{``the journey times on all the routes actually used are equal, and less than those which would be experienced by a single vehicle on any unused route''} \cite{correa2010wardrop}.
 
 \subsection{Finite number of drivers}
\label{gamedef}
 Here, we want to study the relation between  NE of the game in the limit of infinite number of drivers and actual game with  finite number of drivers.
\begin{definition}
A set of strategies $\{Q^{MFE}_{n,l,t}\}$ is said to be an MFE if the following conditions are satisfied.
\begin{itemize}
  \item They are symmetric i.e., 
  \begin{equation*}
 \begin{split}
     \{Q^{MFE}_{1,l,t}\}=\{Q^{MFE}_{2,l,t}\}=\dots=\{Q^{MFE}_{N_l,l,t}\}\quad \forall l,t\\
 \end{split}     
  \end{equation*}
  \item There exists a sequence $\{\epsilon_{N_l}\}$ a satisfying $\epsilon_{N_l} \searrow 0$ as
$N_l \rightarrow \infty$  with constant ratio of $\frac{N_l}{N}$ such that for every possible set of $(n,l)$ and any allowable ${Q}_{n,l,t} \in \mathcal{Q}$,
\begin{equation*}
\begin{split}
        &J_{(n,l)}(Q_{n,l},Q^{MFE}_{-(n,l)})+\epsilon_{N_l} \geq J_{(n,l)}(Q^{MFE}_{n,l},Q^{MFE}_{-(n,l)})
\end{split}
\end{equation*}
\end{itemize}
\end{definition}
\vspace{2mm}
 
 Theorem \ref{theo2} summarizes our next main result which implies that the solution to (\ref{aux}) is indeed MFE of the actual game.% For completeness, the proof is provided in the following.
\begin{theorem}
\label{theo2}
The symmetric strategy profile 
$Q^{ij}_{n,l,t} = Q^{ij*}_{l,t} \  \forall n\leq N_l$ where ${Q}^{ij*}_{l,t}$ are obtained
by Algorithm \ref{alg1}, is the MFE of the road traffic game.
\end{theorem}
\begin{proof}
The proof is similar to the proof provided in \cite[Theorem 2]{tanaka2018linearly}.
% Assume $Q_{-(n,l)}=Q^{*}_{-(n,l)}$. We show that there exist some sequences
%$\epsilon_{N_l} \searrow 0$ as $N \rightarrow \infty$ such that the cost of adopting the strategy $Q^{ij}_{n,l,t} = Q^{ij*}_{l,t}$ for driver $(n,l)$ is no greater than the cost of adopting any other policy plus $\epsilon_{N_l}$. Lemma \ref{lemma1} implies that
%there exists a sequence $\delta_{N_l} \searrow 0$ such that
%\begin{equation*}
 %\Pi^{ij}_{n,l,t}+ \delta_{N_l} >\sum_{m=1}^L a_{lm} \log \frac{Q_{m,t}^{ij*}}{R_{t}^{ij}} \ \  \forall i,j,t.
%\end{equation*}
%Now, for all feasible policy $\{Q_{n,l,t}\}$ for driver $(n,l)$ and the induced distributions $P^{j}_{n,l,t+1}=\sum_{i} P_{n,l,t}^i Q_{n,l,t}^{ij} $, we have
%\begin{equation*}
%\begin{split}
%&\sum_{t=0}^{T-1} \sum_{i, j} P_{n,l,t}^i  Q_{n,l,t}^{ij}(C^{ij}_{n,l,t}+ \hat{\Pi}^{ij}_{N,n,t})\\
%&>\sum_{t=0}^{T-1} \sum_{i, j}P_{n,l,t}^i  Q_{n,l,t}^{ij}(C_{n,l,t}^{ij}+\sum_{m=1}^L a_{lm} \log \frac{Q_{m,t}^{ij*}}{R_{t}^{ij}}-\delta_{N_l})\\
%& =\sum_{t=0}^{T-1}\!\sum_{i, j}\!P_{n,l,t}^i  Q_{n,l,t}^{ij}(C_{n,l,t}^{ij}\!+\!\sum_{m=1}^L a_{lm} \log \frac{Q_{m,t}^{ij*}}{R_{t}^{ij}})\!-\!T V^2 \delta_{N_l}\\
%&\geq \min_{\{Q_{n,l,t}\}_{t\in \mathcal{T}}}\!\sum_{t=0}^{T-1} \sum_{i, j} P_{n,l,t}^i  Q_{n,l,t}^{ij}({C}_{n,l,t}^{ij}\!+\!\!\sum_{m=1}^L\!a_{lm} \log \frac{Q_{m,t}^{ij*}}{R_{t}^{ij}})\\
%&  -T V^2 \delta_{N_l}.
%\end{split}
%\end{equation*}
%The minimization in the last line is established by assuming $Q^{ij}_{n,l,t} =Q^{ij*}_{l,t}$. Adopting $\epsilon_{N_l} =T V^2 \delta_{N_l} \searrow 0$
%completes the proof.
\end{proof}

\section{Numerical Illustartion}
\label{NI}
Here, we consider an example of routing game over a traffic network (a grid world with obstacles). Team 1 and 2 are concentrated in the origin cells (indicated by $1S$ and $2S$) at $t=0$.
For $t \in \mathcal{T} $, the travel cost for each driver is given by:
\[
  C^{ij}_{l,t}=\begin{cases}
   C_{l,\text{term}} \quad \quad \quad \quad \quad \text{if} \ j=i\\
   1+C_{l,\text{term}} \quad \quad \  \quad  \text{if} \  j \in \mathcal{V}(i) \text{ and $i\neq j$}\\
        100000+C_{l,\text{term}} \quad  \text{if} \  j \not\in \mathcal{V}(i) \text{ or $j$ is an obstacle}
            \end{cases}
\]
where $\mathcal{V}(i)$ contains the cell $i$ itself and its north, east, south, and west neighboring cells. We
introduce $C_{\text{term}} = 0$ if $t \neq T-1$ and $C_{\text{term}} = 10\sqrt{\text{dist}(j, D)}$ if $t = T-1$ to penalize drivers if they end up far from their targeted destination, where $\text{dist}(j, D)$ is the Manhattan distance between the driver’s final location $j$ and the destination cell (indicated by $1E$ and $2E$). As the desired distribution, we use $R^{ij}_t= \frac{1}{|\mathcal{V}(i)|}$ (uniform distribution) for each $i \in \mathcal{V}$ and $t \in \mathcal{T}$ to incentivize drivers to avoid concentrations.
\begin{figure}[thpb]
      \centering
      \includegraphics[scale=0.50]{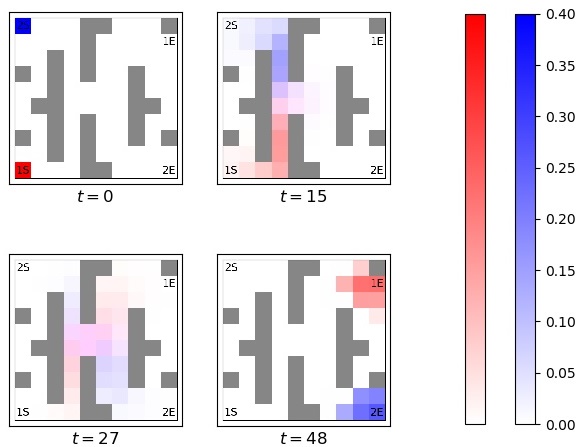}
      \caption{ Density propagation of team one (red) and team two (blue) for $a_{11}=a_{22}=3$ and $a_{12}=a_{21}=2$.}
      \label{fig1}
   \end{figure} 
   \begin{figure}[thpb]
      \centering
      \includegraphics[scale=0.50]{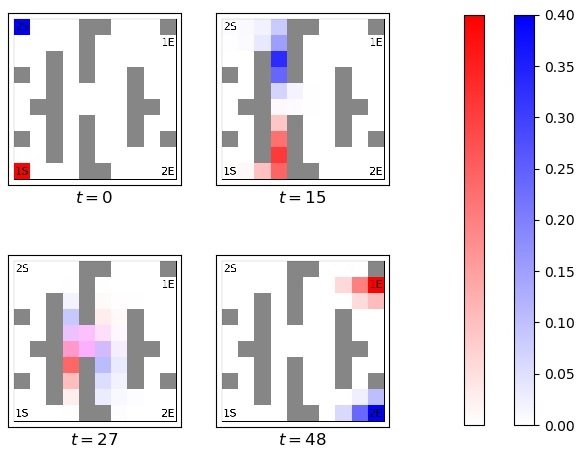}
      \caption{Density propagation of team one (red) and team two (blue) for $a_{11}=a_{22}=0.06$ and $a_{12}=a_{21}=0.04$.}
      \label{fig2}
   \end{figure} 
   
For various values of $A$, the optimal policy is computed based on Algorithm \ref{alg1} for $T=50$. Fig. \ref{fig1} ($a_{11}=a_{22}=3$ and $a_{12}=a_{21}=2$) and Fig. \ref{fig2} ($a_{11}=a_{22}=0.06$ and $a_{12}=a_{21}=0.04$) show the snapshot of population distributions of teams 1 and 2 (respectively in the red and blue) for intermediate time steps of $t=0, 15,27$, and $48$. As expected, when $a_{ij}$s are small and the fixed cost dominates, drivers will stay concentrated with higher level of congestion.
Conversely, when $a_{ij}$s are large and the congestion associated tax is dominant, they would prefer to spread more to lower the level of congestion.
%%%%%%%%%%%%%%%%%%%%%%%%%%%%%%%%%%%%%%%%%%%%%%%%%%%%%%%%%%%%%%%%%%%%%%%%%%%%%%%%
\section{Conclusion and Future Work}
\label{conc}
In this paper, a log-population tax setting for the multi-team routing game is proposed. Under this toll mechanism, an approximation of the game in the large population limit is obtained. It is shown the NE of the approximated game can be efficiently computed by a linearly-solvable algorithm. Furthermore, it is shown that the NE of approximated game is indeed an MFE of a large but finite population game.

Providing a protocol to find the suitable values of mechanism's parameters ($R^{ij}_t$ and $a_{ij}$) is left unanswered, which is worth investigating in future work. Practical application of the results for congestion management in urban traffic networks also remains to be explored. We  also plan to study the proposed framework from the viewpoint of mechanism design theory. 

%%%%%%%%%%%%%%%%%%%%%%%%%%%%%%%%%%%%%%%%%%%%%%%%%%%%%%%%%%%%%%%%%%%%%%%%%%%%%%%%
\section{ACKNOWLEDGMENTS}

The authors would like to thank Samuel R. Faulkner
at the University of Texas at Austin for his contributions to the numerical study in Section \ref{NI}.

%%%%%%%%%%%%%%%%%%%%%%%%%%%%%%%%%%%%%%%%%%%%%%%%%%%%%%%%%%%%%%%%%%%%%%%%%%%%%%%%

\bibliographystyle{IEEEtran}
\bibliography{ref}

%%%%%%%%%%%%%%%%%%%%%%%%%%%%%%%%%%%%%%%%%%%%%%%%%%%%%%%%%%%%%%%%%%%%%%%%%%%%%%%%
\appendix
\setlength{\belowdisplayskip}{1pt} \setlength{\belowdisplayshortskip}{1pt}
\setlength{\abovedisplayskip}{1pt} \setlength{\abovedisplayshortskip}{1pt}
\subsection{Proof of Lemma \ref{lemma1}}
\label{appen1}
Consider a single term $m\neq l$. Then:
\begin{equation*}
    \begin{split}
       &\mathbb{E} \Big[ \log( \frac{K^{ij}_{m,t}}{K^i_{m,t}})-
       \log R^{ij}_t\Big]=\\
       &\mathbb{E} \Big[ \log \frac{K^{ij}_{m,t}}{N_m}\Big]-
       \mathbb{E} \Big[ \log \frac{K^{i}_{m,t}}{N_m}\Big]- \log R^{ij}_t.
    \end{split}
\end{equation*}
It is elementary to check that $\mathbb{E} \Big[ \frac{K^{ij}_{m,t}}{N_m}\Big] =P^{i*}_{m,t} Q^{ij*}_{m,t}$ and $ \mathbb{E} \Big[ \frac{K^{i}_{m,t}}{N_m}\Big] =P^{i*}_{m,t}$. However, since $\log(.)$ is an unbounded function, we cannot directly deploy continuous mapping theorem \cite{van2000asymptotic}. In the following, we show that if the general random variable $\frac{K_N}{N} \xrightarrow{a.s.} p$ has bounded variance $\sigma^2_N$, then $ \lim_{N\rightarrow \infty } \mathbb{E} \log(\frac{K_N}{N}) \rightarrow \log(p)$.
Let $\delta$ be a number satisfying $0 < \delta < p$. Then, we can write $\log x = f_1(x) + f_2(x)$ where 
\begin{equation*}
    \begin{split}
        &f_1(x) = \max \{\log x, \log \delta\} \\ 
        &f_2(x) = 1_{\{x\leq\delta\}} \log x.
    \end{split}
\end{equation*}
We first show $\lim_{N\rightarrow\infty} \mathbb{E} f_1(\frac{K_N}{N})= \log p$. By the strong law of large numbers, we have $ \frac{K_N}{N} \xrightarrow{a.s} p $. By the continuous mapping theorem, $f_1(\frac{K_N}{N}) \xrightarrow{a.s.}
f_1(p)$. By the bounded convergence theorem, we have
\begin{equation*}
   \mathbb{E} f_1(\frac{K_N}{N}) \rightarrow  f_1(p)\quad  \text{as} \quad  N\rightarrow \infty .
\end{equation*}
Next, we show $\lim_{N\rightarrow \infty} \mathbb{E} f_2(\frac{KN}{N})= 0$. Notice that $\frac{K_N}{N}$
has mean $\mu_N = p$
and covariance $\sigma^2_N=\frac{p(1-p)}{N}$. We have
\begin{subequations}
    \begin{align}
    \nonumber
&\mathbb{E}\Big |f_2(\frac{K_N}{N}) \Big|= \mathbb{E} 1_{\{\frac{K_N}{N}\leq \delta\}}
\Big|\log(\frac{K_N}{N})- \log \delta \Big|\\
\label{a}
&\leq \mathbb{E}1_{\{\frac{K_N}{N}\leq \delta\}}
\Big(\Big|\log(\frac{K_N}{N})\Big|+ \Big|\log \delta\Big|\Big)\\
\label{b}
&\leq \mathbb{E}1_{\{\frac{K_N}{N}\leq \delta\}}
(|\log N|+ |\log \delta |)\\
\nonumber
&= \textbf{Pr}(\frac{K_N}{N}\leq \delta) (|\log N|+ |\log \delta|)\\
\nonumber
&= \textbf{Pr}(\frac{K_N}{N}-\mu_N\leq \delta-p) (|\log N|+ |\log \delta|)\\
\label{c}
&= \textbf{Pr}(|\frac{K_N}{N}-\mu_N|\geq p-\delta) (|\log N|+ |\log \delta|)\\
\label{d}
&\leq \frac{p(1-p)}{N( p-\delta)}(|\log N |+ |\log \delta|)\\
\nonumber
&\rightarrow 0 \quad \text{as} \quad N \rightarrow \infty.
     \end{align}
\end{subequations}
From (\ref{a}) to (\ref{b}), we used the fact that $1 \leq K_N \leq N$ and thus $|\log \frac{K_N}{N}|$ is maximized when $K_N = 1$. From (\ref{c}) to (\ref{d}), the Chebyshev inequality
\begin{equation*}
    \begin{split}
    \textbf{Pr}(|\frac{K_N}{N}-\mu_N \geq \epsilon|)\leq \frac{\sigma^2_N}{\epsilon}
    \end{split}
\end{equation*}
is used with $\epsilon =p-\delta$
. Finally,
\begin{equation*}
    \begin{split}
&\lim_{N\rightarrow \infty} \mathbb{E} \log \frac{K_N}{N}=\\
&\lim_
{N\rightarrow \infty} \mathbb{E} f_1(\frac{K_N}{N})+
\lim_{N\rightarrow \infty}
\mathbb{E} f_2(\frac{K_N}{N})= \log p.
 \end{split}
\end{equation*}
 For $m=l$, the same method can be repeated by $\mu_N=p-\frac{1}{N}$ and $\sigma^2_N=\frac{p(1-p)}{N-1}$.
%%%%%%%%%%%%%%%%%%%%%%%%%%%%%%%%%%%%%%%%%%%%%%%%%%%%%%%%%%%%%%%%%%%%%%%%%%%%%%%%
\subsection{Proof of Theorem \ref{theorem1}}
\label{appen2}
For the base step of backward induction, let's define $\phi^{ij}_{l,T-1}= C^{ij}_{l,T-1}$. The Bellman equation (\ref{Bellman}) for $t=T-1$ reduces to:
\begin{align}
\nonumber\min_{Q_{l,T-1}^{ij}} & 
 \sum_{i,j} P_{l,T-1}^i \  Q_{l,T-1}^{ij}(\phi_{l,T-1}^{ij}+\sum_{m=1}^L a_{lm} \log \frac{Q_{m,T-1}^{ij}}{R_{T-1}^{ij}}) \\\label{base}
 \text{s.t } \quad & \sum_{j}  Q_{l,T-1}^{ij} =1\quad \forall i.
\end{align}
The associated Lagrangian is:
\begin{equation*}
\label{Lagrangian}
\begin{split}
&\mathcal{L}_{l,T-1}(Q^{ij}_{l,T-1}) =\\
&\sum_{i, j}P_{l,T-1}^i  Q_{l,T-1}^{ij}(\phi_{l,T-1}^{ij}+\sum_{m=1}^L a_{lm} \log \frac{Q_{m,T-1}^{ij}}{R_{T-1}^{ij}})\\
& +\sum_{i} P^i_{l,T-1} \lambda^i_{l,T-1}(\sum_{j} Q^{ij}_{l,T-1} -1).
\end{split}
\end{equation*}
The optimal solution ${Q}^{*}$ satisfies the stationarity condition:
\begin{equation}
\label{stationarity}
\begin{aligned}
\frac{\partial \mathcal{L}_{l,T-1}}{\partial{Q}^{ij}_{l,T-1}}
\Big|_{Q^{*}}&\!\!\!\!={P}^i_{l,T-1}(\phi_{l,T-1}^{ij}+\!\sum_{m=1}^L a_{lm} \log \frac{Q_{m,T-1}^{ij*}}{R_{T-1}^{ij}}\!)\\
&+ {P}^i_{l,T-1} a_{ll}+ P^i_{l,T-1} \lambda^i_{l,T-1}=0.
\end{aligned}
\end{equation}
Equivalently:
\begin{equation*}
\label{st_3}
\begin{bmatrix} 
\log(\frac{Q_{1,T-1}^{ij*}}{R_{T-1}^{ij}})\\
\vdots\\
\log(\frac{Q_{L,T-1}^{ij*}}{R_{T-1}^{ij}})
\end{bmatrix}=\overbrace{ A^{-1}
\begin{bmatrix}
-a_{11}-\phi^{ij}_{1,T-1}\\
\vdots\\
-a_{LL}-\phi^{ij}_{L,T-1}
\end{bmatrix}}^{M^{ij}_{T-1}}-A^{-1}\overbrace{
\begin{bmatrix}
\lambda^{i}_{1, T-1}\\
\vdots\\
\lambda^{i}_{L,T-1}
\end{bmatrix}}^{\Lambda^{i}_{T-1}}.
\end{equation*}
 If we exponentiate both sides element-wise, we have:
 \begin{equation}
 \label{qq}
     Q_{l,T-1}^{ij*}=R_{T-1}^{ij} \exp{M^{ij}_{l,T-1}}\exp{([- A^{-1}\Lambda^{i}_{T-1}]_l)}.
 \end{equation}
We can use the fact that $\sum_j Q_{l,T-1}^{ij*}=1$ to see
\begin{equation*}
[A^{-1}\Lambda^{i}_{T-1}]_l=
     \log(\sum_j R_{T-1}^{ij} \exp{M^{ij}_{l,T-1}}).
 \end{equation*}
Consequently,
\begin{equation*}
\Lambda^{i}_{T-1}= A \begin{bmatrix}
\log(\sum_j R_{T-1}^{ij} \exp{M^{ij}_{1,T-1}})\\
\vdots\\
\log(\sum_j R_{T-1}^{ij} \exp{M^{ij}_{L,T-1}})
\end{bmatrix},
 \end{equation*}
 By substituting $\Lambda^{i}_{T-1}$ in (\ref{qq}), $Q_{l,T-1}^{ij*}$ is obtained explicitly. Furthermore, (\ref{stationarity}) implies that:
\begin{align}
\label{q-star}
  \phi_{l,T-1}^{ij}+\sum_{m=1}^L a_{lm} \log \frac{Q_{m,T-1}^{ij*}}{R_{T-1}^{ij}}= -a_{ll}-\lambda^i_{l,T-1}.
\end{align}
 Therefore, by definition:
 \begin{subequations}
 \begin{align}
 \nonumber
&V_{l,T-1}(P_{l,T-1})\\\nonumber
&=\sum_{ij}P^{i}_{l,T-1} Q_{l,T-1}^{ij*} (\phi_{l,T-1}^{ij}+\sum_{m=1}^L a_{lm} \log \frac{Q_{m,T-1}^{ij*}}{R_{T-1}^{ij}})\\\nonumber
 &=\sum_{ij}P^{i}_{l,T-1} Q_{l,T-1}^{ij*} (-a_{ll}-\lambda^i_{l,T-1})\\\nonumber
 &=\sum_{i}P^{i}_{l,T-1}  (-a_{ll}-\lambda^i_{l,T-1})\underbrace{\sum_{j} Q_{l,T-1}^{ij*}}_{=1}\\ \nonumber
 & =\sum_{i}P^{i}_{l,T-1}  (-a_{ll}-\lambda^i_{l,T-1}).
 \end{align}
 \end{subequations}
 Plugging this value in the Bellman equation for $t=T-2$ leads to a similar optimization to (\ref{base}) with $\phi^{ij}_{l,T-2}=C^{ij}_{l,T-2}-a_{ll}-{\lambda}^j_{l,T-1}$ which completes the proof. 

\end{document}